\documentclass[11pt,a4paper]{article}
\usepackage{amssymb,amsmath,amsopn,amsthm,graphicx,makeidx}
\usepackage[all,2cell]{xy} \UseAllTwocells

\usepackage{geometry}
\begin{document}

\newtheorem{definition}{Definition}[section]
\newtheorem{definitions}[definition]{Definitions}
\newtheorem{lemma}[definition]{Lemma}
\newtheorem{proposition}[definition]{Proposition}
\newtheorem{theorem}[definition]{Theorem}
\newtheorem{cor}[definition]{Corollary}
\newtheorem{cors}[definition]{Corollaries}
\theoremstyle{remark}
\newtheorem{remark}[definition]{Remark}
\theoremstyle{remark}
\newtheorem{remarks}[definition]{Remarks}
\theoremstyle{remark}
\newtheorem{notation}[definition]{Notation}
\theoremstyle{remark}
\newtheorem{example}[definition]{Example}
\theoremstyle{remark}
\newtheorem{examples}[definition]{Examples}
\theoremstyle{remark}
\newtheorem{dgram}[definition]{Diagram}
\theoremstyle{remark}
\newtheorem{fact}[definition]{Fact}
\theoremstyle{remark}
\newtheorem{illust}[definition]{Illustration}
\theoremstyle{remark}
\newtheorem{rmk}[definition]{Remark}
\theoremstyle{definition}
\newtheorem{question}[definition]{Question}
\theoremstyle{definition}
\newtheorem{conj}[definition]{Conjecture}

\newcommand{\mcal}[1]{\mathcal{#1}}
\newcommand{\Leq}{{\mathbb L}^{\rm eq+}}
\renewcommand{\phi}{\varphi}
\renewcommand{\marginpar}[2][]{}
\renewenvironment{proof}{\noindent {\bf{Proof.}}}{\hspace*{3mm}{$\Box$}{\vspace{9pt}}}

\title{Representation embeddings, interpretation functors and controlled wild algebras}

\author{Lorna Gregory and Mike Prest\footnote{lorna.a.gregory@manchester.ac.uk, mprest@manchester.ac.uk, School of Mathematics, Alan Turing Building, University of Manchester, Manchester M13 9PL, UK \newline The authors acknowledge the support of EPSRC through Grant EP/K022490/1.}}

\maketitle

\abstract{We establish a number of results which say, roughly, that interpretation functors preserve algebraic complexity.

First we show that representation embeddings between categories of modules of finite-dimensional algebras induce embeddings of lattices of pp formulas and hence are non-decreasing on Krull-Gabriel dimension and uniserial dimension. A consequence is that the category of modules of any wild finite-dimensional algebra has width $\infty$ and hence, if the algebra is countable, there is a superdecomposable pure-injective representation.

It is conjectured that a stronger result is true:  that a representation embedding from ${\rm Mod}\mbox{-}S$ to ${\rm Mod}\mbox{-}R$ admits an inverse interpretation functor from its image and hence that, in this case, ${\rm Mod}\mbox{-}R$ interprets ${\rm Mod}\mbox{-}S$.  This would imply, for instance, that every wild category of modules interprets the (undecidable) word problem for (semi)groups.  We show that the conjecture holds for finitely controlled representation embeddings.

Finally we prove that if $R,S$ are finite dimensional algebras over an algebraically closed field and $I:{\rm Mod}\mbox{-}R\rightarrow{\rm Mod}\mbox{-}S$ is an interpretation functor such that the smallest definable subcategory containing the image of $I$ is the whole of ${\rm Mod}\mbox{-}S$ then, if $R$ is tame, so is $S$ and similarly, if $R$ is domestic, then $S$ also is domestic.}\footnote{MSC-class:  16G60 (primary) 03C60, 03D35, 16G20, 16D90 (Secondary)}

\maketitle


\section{Introduction} 
\label{intro}

Suppose that $R$, $S$ are rings.  Let ${\rm mod}\mbox{-}R$, respectively ${\rm Mod}\mbox{-}R$, denote the category of finitely presented, respectively all, right $R$-modules.  An {\bf interpretation functor} from ${\rm Mod}\mbox{-}R$ to ${\rm Mod}\mbox{-}S$ can be defined most easily as an additive functor which commutes with direct products and direct limits but the original definition, and the one from which the name arises, is that it is an interpretation in the model-theoretic sense, given, since we are in the additive context, by pp formulas; see \cite{PreInterp}, \cite[Chpts.~13, 25]{PreMAMS}) and also Section \ref{secinterp} here.  We consider, more generally, interpretation functors between {\bf definable categories}, that is, categories equivalent to {\bf definable subcategories} of module categories:  these are the subcategories closed under direct products, direct limits and pure submodules (and isomorphism).  Again, the name refers to their model-theoretic definition.

Representation embeddings are examples of interpretation functors.  Recall the following result of Eilenberg and Watts \cite{Eilen}, \cite{Watt}.

\begin{theorem}\label{ew}\marginpar{ex} Suppose that $F:{\rm Mod}\mbox{-}S \rightarrow {\rm Mod}\mbox{-}R$ is an additive, right exact functor which commutes with direct sums (equivalently, is additive and has a right adjoint).  Then $F$ is isomorphic to the functor $-\otimes_S(FS)_R$.

If $F$ is exact then $FS$ must be flat as a left $S$-module.
\end{theorem}

We are particularly interested in the case that $R$ and $S$ are finite-dimensional algebras and that $ -\otimes _SB_R: {\rm Mod}\mbox{-}S\rightarrow  {\rm Mod}\mbox{-}R $ is a {\bf representation embedding}, meaning that $ -\otimes B $ is exact and restricts to a functor (which we also refer to as a representation embedding) $ {\rm mod}\mbox{-}S\rightarrow {\rm mod}\mbox{-}R$ and this restricted functor preserves indecomposability and reflects isomorphism.  The intuition behind this definition is that it indicates that the category of (finite-dimensional) $R$-modules is at least as complex, in the sense of obtaining some classification of modules, as that of (finite-dimensional) $S$-modules.  One might expect, or even conjecture, that this particular meaning of ``complexity" will correspond to other ways of measuring the complexity of an abelian category.  Our results support this expectation, and give a partial answer to a long-standing conjecture of the second author (\cite[p.~350]{PreBk}) that ``wild implies undecidable".  This conjecture has been proved for strictly wild (and somewhat more general, \cite{PreEpi}) algebras and has been verified in many particular cases, both within and outwith the context of finite-dimensional algebras (see \cite[p.~651ff.]{PreNBK} for a relatively recent summary and references); here we prove it for controlled wild algebras.  Our results on non-decreasing complexity apply to any dimension which can be defined in terms of the lattice of pp formulas, equivalently, to any dimension which can be defined in terms of successive localisations of functor categories $({\rm mod}\mbox{-}R, {\bf Ab})^{\rm fp}$ (the category of finitely presented additive functors from the category, ${\rm mod}\mbox{-}R$ of finitely presented $R$-modules to the category, ${\bf Ab}$, of abelian groups).

In Section \ref{seclat} we show that a representation embedding induces an embedding of pp-lattices, hence these functors are non-decreasing on dimensions as above.  The method that we use there is very explicit, though we will see, in the second half of Section \ref{secinterp}, that using deeper results from the general theory allows a faster proof at least for finite-dimensional algebras. We do this by showing that if $R,S$ are finite-dimensional $k$-algebras and $I:{\rm Mod}\mbox{-}R\rightarrow {\rm Mod}\mbox{-}S$ is an interpretation functor whose image generates ${\rm Mod}\mbox{-}S$ as a definable subcategory then $I$ induces an embedding of pp-lattices. From a result of Nagase \cite[Proposition 2.3]{Nag}, it follows that the image of the right adjoint $I:{\rm Mod}\mbox{-}S\rightarrow {\rm Mod}\mbox{-}R$ of a representation embedding $F:{\rm mod}\mbox{-}R\rightarrow {\rm mod}\mbox{-}S$ (extended to the whole of ${\rm Mod}\mbox{-}R$), which is an interpretation functor, generates ${\rm Mod}\mbox{-}R$ as a definable subcategory.

In Section \ref{seccontr} we show that finitely controlled representation embeddings are reversible; this allows us to conclude that over any finitely controlled wild algebra the theory of modules interprets the word problem for groups, hence is undecidable.

Interpretation functors are much more general than representation embeddings but there is a corresponding underlying idea in that, if there is an interpretation functor $I$ from a definable category ${\cal D}$ to another ${\cal C}$ such that the definable subcategory generated by the image of $I$ is all of ${\cal C}$, then, roughly, ${\cal D}$ `contains' ${\cal C}$ in a model-theoretic sense (this is made precise using the model-theoretic terminology of sorts and imaginaries).  So if ${\cal D}$ is a definable subcategory of ${\rm Mod}\mbox{-}R$ and there is an interpretation functor from ${\cal D}$ to ${\rm Mod}\mbox{-}S$ such that the definable subcategory generated by the image of $I$ is all of ${\rm Mod}\mbox{-}S$, then we may regard ${\cal D}$, and hence ${\rm Mod}\mbox{-}R$, as `containing' ${\rm Mod}\mbox{-}S$ in a model-theoretic sense.  Therefore, if model-theoretic and algebraic measures of complexity are to align, then it should not be possible for a tame category of modules to interpret one which is wild, or for a domestic category of modules to interpret one which is non-domestic.  We prove this, in the case that ${\cal D} = {\rm Mod}\mbox{-}R$, for finite-dimensional algebras over algebraically closed fields, in Section 5 and 6 respectively.

Throughout, $k$ is a field, all algebras are $k$-algebras, all functors are assumed to be additive and, where it makes sense, $k$-linear (so, below, we could as well write $({\rm mod}\mbox{-}R, {\rm Mod}\mbox{-}k)^{\rm fp}$ in place of the equivalent functor category $({\rm mod}\mbox{-}R, {\bf Ab})^{\rm fp}$).

\section{Lattice embeddings}\label{seclat}

We will express the results here using the language and techniques of pp formulas (in $n$ free variables) for $R$-modules and the lattice, ${\rm pp}_R^n$, that they form.  There are many sources that explain this, for instance \cite{PreNBK}.  We recall that the ordering\footnote{rather, preodering, but we don't distinguish between equivalent pp formulas, that is, formulas with the same solution set in every module} on pp formulas (in the same free variables) is implication, equivalently inclusion of their solution sets:  $\psi \leq \phi$ if $\psi(M) \leq \phi(M)$ for every (finitely presented) module $M$.  The lattice ${\rm pp}_R^n$ is naturally isomorphic to the lattice of finitely generated subfunctors of the $n$-th power, $(R^n,-)$, of the forgetful functor from ${\rm mod}\mbox{-}R$ to ${\bf Ab}$ (\cite[10.2.33]{PreNBK}), the isomorphism being given by $\phi \mapsto F_\phi$ where the latter is the functor that assigns to every finitely presented module $M$, the solution set $\phi(M)$.  Every functor $F\in ({\rm mod}\mbox{-}R, {\bf Ab})$  has an essentially unique extension to a functor from ${\rm Mod}\mbox{-}R$ to ${\bf Ab}$ which commutes with direct limits; we denote this extension by $\overrightarrow{F}$, or just $F$, and we will use the term {\bf coherent functor} for such extensions.  Since the functor $M\mapsto \phi(M)$ commutes with direct limits, we can safely use $F_\phi$ to denote the solution-set functor on all modules or its restriction to finitely presented modules.  In view of this identification of pp formulas and finitely presented subfunctors of powers of the forgetful functor, \ref{latt} below says that a representation embedding $ {\rm Mod}\mbox{-}S\rightarrow  {\rm Mod}\mbox{-}R $ induces an embedding from the lattice of finitely generated subfunctors of $(S,-)$ to that of $(R^n,-)$ for a suitable $n$.

\vspace{4pt}

Suppose that $R$ and $S$ are rings and that $ -\otimes _SB_R: {\rm Mod}\mbox{-}S\rightarrow  {\rm Mod}\mbox{-}R $ sends finitely presented $S$-modules to finitely presented $R$-modules (equivalently $B_R$ is finitely presented). Choose a finite tuple $ \overline{t} =(t_1,\dots, t_n) $ from $ B $ which generates the module $ B_R$. We use this tuple to define a map $ \beta :{\rm pp}^1_S\rightarrow {\rm pp}^n_R $, where $ n $ is the length of $ \overline{t}$, as follows.

Given $ \phi \in {\rm pp}^1_S $ choose a {\bf free realisation} $ (C,c) $ of $ \phi $; that is, $C$ is finitely presented and the {\bf pp-type} of $c$ in $C$, ${\rm pp}^C(c) =\{ \psi \in {\rm pp}^1_R: c\in \psi(C)\}$, is {\bf generated by} $\phi$, meaning that it is the set of all pp formulas $\psi \geq \phi$ (e.g., see \cite[\S 1.2.2]{PreNBK}). Consider the tuple $ c\otimes \overline{t} = (c\otimes t_1, \dots, c\otimes t_n)$ from $ C\otimes B$: this module is finitely presented over $ R $ so we define $ \beta \phi  =\beta _{\overline{t}} \phi$ to be any pp formula in $ {\rm pp}^n_R $ which generates the pp-type of $ c\otimes \overline{t}$ in $ C\otimes B$. The choice of formula is not literally unique but, up to equivalence on $R$-modules, that is, up to equality in $ {\rm pp}^n_R$, it is.  So this map $\beta$ is well-defined, given $(C,c)$.

This map also is independent of choice of free realisation of $ \phi  $ since, if $ (D,d) $ is another free realisation of $ \phi  $, then there are morphisms $ f:C\rightarrow D $ and $ g:D\rightarrow C $ such that $ gfc=c $ and $ fgd=d $ so then, applying $ -\otimes B$, we obtain morphisms $ f'=f\otimes 1_B:C\otimes B\rightarrow D\otimes B $ and $ g'=g\otimes 1_B:D\otimes B\rightarrow C\otimes B $ with $ g'f'(c\otimes \overline{t})=c\otimes \overline{t} $ and $ f'g'(d\otimes \overline{t})=d\otimes \overline{t}$.  Therefore (by \cite[1.2.8]{PreNBK}) $ {\rm pp}^{C\otimes B}(c\otimes \overline{t})={\rm pp}^{D\otimes B}(d\otimes \overline{t})$, as claimed.

The above construction is based on Harland's explicit proofs on the effect of tilting functors - these are interpretation functors - in \cite[\S 4.7.1]{Har}.

Recall that a ring is said to be {\bf Krull-Schmidt} if every finitely presented (right or left - it makes no difference, e.g.~see \cite[\S 4.3.8]{PreNBK}) module is a finite direct sum of indecomposable modules, each with local endomorphism ring; so the isomorphism types of these summands, including their multiplicities, are uniquely determined.

\begin{theorem}\label{latt}\marginpar{latt} Let the bimodule $_SB_R$ be such that $B_R$ is finitely presented.  Choose a finite generating ($n$-)tuple $\overline{t}$ for $B_R$.  Then the induced map $ \beta_{\overline{t}}: {\rm pp}^1_S\rightarrow {\rm pp}^n_R $ is a homomorphism of lattices.

If $R,S$ are Krull-Schmidt and $-\otimes _SB_R$ is a representation embedding, meaning that it has the properties stated in the introduction, then $\beta_{\overline{t}}$ is an embedding of lattices.
\end{theorem}
\begin{proof}  The map $ \beta  $ is order-preserving since, if $ \psi \leq \phi  $ are pp formulas with respective free realisations $ (C_\psi ,c_\psi ) $ and $(C_\phi ,c_\phi ) $, then (\cite[1.2.21]{PreNBK}) there is a morphism $ f:C_\phi \rightarrow C_\psi  $ taking $ c_\phi  $ to $ c_\psi $. Tensoring with $ B $ gives $$ f\otimes 1_B:C_\phi \otimes B\rightarrow C_\psi \otimes B $$ taking $ c_\phi \otimes \overline{t} $ to $ c_\psi \otimes \overline{t}$, hence $$ {\rm pp}^{C_\phi \otimes B}(c_\phi \otimes \overline{t})\geq {\rm pp}^{C_\psi \otimes B}(c_\psi \otimes \overline{t}) $$ (by which we mean $ {\rm pp}^{C_\phi \otimes B}(c_\phi \otimes \overline{t}) \subseteq {\rm pp}^{C_\psi \otimes B}(c_\psi \otimes \overline{t})$) and so $ \beta \phi \leq \beta \psi $.

The map $ \beta  $ preserves the lattice operation $ + $ (that is, sup): take $ \phi ,\psi \in {\rm pp}^1_S$; then, with notation as above, $ (C_\phi \oplus C_\psi ,(c_\phi, c_\psi) ) $ is a free realisation of $ \phi +\psi $ (\cite[1.2.27]{PreNBK}). By definition $ \beta (\phi +\psi ) $ is a generator of $$ {\rm pp}^{(C_\phi \oplus C_\psi )\otimes B}((c_\phi, c_\psi )\otimes \overline{t}) ={\rm pp}^{(C_\phi \otimes B) \oplus (C_\psi \otimes B)}(c_\phi \otimes \overline{t}, c_\psi \otimes \overline{t})$$ $$={\rm pp}^{C_\phi \otimes B}(c_\phi \otimes \overline{t})+{\rm pp}^{C_\psi \otimes B}(c_\psi \otimes \overline{t}) $$ which is generated by $\beta \phi +\beta \psi$.

The map $ \beta  $ preserves the lattice operation $ \wedge  $ (that is, inf): with notation as above, consider the pushout

$\xymatrix{S \ar[r]^{c_\phi } \ar[d]_{c_\psi } & C_\phi  \ar[d]^f \\ C_\psi  \ar[r]_g & C_{\phi \wedge \psi } }$

\noindent where by the map $ c_\phi  $ we mean that which takes $ 1_S $ to $ c_\phi $ and similarly for $c_\psi$. Recall (\cite[1.2.28]{PreNBK}) that $ (C_{\phi \wedge \psi },c)$, where $ c=fc_\phi =gc_\psi $, is, as the notation indicates, a free realisation of $ \phi \wedge \psi $. Since the functor $ -\otimes B $ is a left adjoint, so preserves colimits, the diagram

$\xymatrix{B \ar[r]^{c_\phi \otimes 1_B} \ar[d]_{c_\psi \otimes 1_B} & C_\phi \otimes B \ar[d]^{f\otimes 1_B} \\ C_\psi \otimes B \ar[r]_{g\otimes 1_B} & C_{\phi \wedge \psi }\otimes B }$

\noindent is a pushout in $ {\rm mod}\mbox{-}R$. Since morphisms preserve pp formulas, both $\beta \phi$ and $\beta \psi$, and hence $ \beta \phi \wedge \beta \psi$, are in ${\rm pp}^{C_{\phi \wedge \psi }\otimes B}(c\otimes \overline{t})$, which is, by definition, generated by $ \beta (\phi \wedge \psi )$.  Therefore $\beta \phi \wedge \beta \psi \geq \beta (\phi \wedge \psi )$. We now show the reverse inequality.  Enlarge the above diagram as shown below where, if $ \overline{t}=(t_1,\dots,t_n)$, then $ \overline{t} $ also denotes the morphism, which by assumption is a surjection, which takes $ e_i=(0,\dots, 1,0,\dots,0) $ ($1$ in the $i$th position) to $ t_i$ and where $ P $ together with $k,h$, is the pushout of $ c_\phi \otimes 1_B \cdot \overline{t} = c_\phi \otimes \overline{t}$ and $  c_\psi \otimes 1_B \cdot\overline{t} = c_\psi \otimes \overline{t}$.

$\xymatrix{R^n \ar[rd]^{\overline{t}} \ar[rrd] \ar[rdd] \\ & B \ar[r]_{c_\phi \otimes 1_B} \ar[d]^{c_\psi \otimes 1_B} & C_\phi \otimes B  \ar[d]_{f\otimes 1_B} \ar[rdd]^h \\ & C_\psi \otimes B \ar[r]^{g\otimes i_B} \ar[rrd]_k & C_{\phi \wedge \psi }\otimes B \ar@{.>}[rd]^l \\ & & & P}$

\noindent Set $ \overline{p}=h \cdot c_\phi \otimes 1_B  \cdot \overline{t} = k \cdot c_\psi \otimes 1_B  \cdot \overline{t} $; that is, $\overline{p} =h(c_\phi \otimes \overline{t}) = k(c_\psi \otimes \overline{t})$ and so $ (P,\overline{p}) $ is a free realisation of $ \beta \phi \wedge \beta \psi $. Since the morphism $ \overline{t} $ is surjective we have $ h\cdot c_\phi \otimes 1_B = k \cdot c_\psi \otimes 1_B$, so there is a morphism $ l $ as shown. Since $ l(c\otimes \overline{t})=\overline{p} $ we conclude that $ \beta (\phi \wedge \psi ) \geq \beta \phi \wedge \beta \psi $, as required.

That shows that we always have a lattice homomorphism.  Now suppose that the functor $-\otimes B$ is a representation embedding and that both $R$ and $S$ are Krull-Schmidt. We must show that if $ \phi >\psi  $ in $ {\rm pp}^1_S$ then $ \beta \phi >\beta \psi $.

First we deal with the case that $ C_\phi  $ is indecomposable.  Choose a morphism $ f:C_\phi \rightarrow C_\psi  $ with $ fc_\phi =c_\psi $. If we had $ \beta \phi =\beta \psi  $ then $ f\otimes 1_B:C_\phi \otimes B\rightarrow C_\psi \otimes B $ would, by assumption, be pp-type-preserving, so there would be $ h:C_\psi \otimes B\rightarrow C_\phi \otimes B $ such that $ h\cdot f\otimes 1_B(c_\phi \otimes \overline{t})=c_\phi \otimes \overline{t}$.

Since $ h\cdot f\otimes 1_B(c_\phi \otimes \overline{t})=c_\phi \otimes \overline{t} $ it must be that $ h\cdot f\otimes 1_B $ is an automorphism of $ C_\phi \otimes B $ and hence $ f\otimes 1_B $ is a split embedding.  Decompose $ C_\psi  $ as a direct sum, $ C_1\oplus \dots \oplus C_m$, of indecomposables, decomposing $ c_\psi  $ and hence (\cite[1.2.27]{PreNBK}) $ \psi  $ accordingly, say $ \psi =\psi _1+\dots +\psi _m$. Since $ \phi >\psi $, also $ \phi >\psi _i $ for each $ i $. By hypothesis, and using the already-established properties of $ \beta $, we have $ \beta \phi =\beta \psi =\sum_i\beta \psi _i$.  Since $ -\otimes B $ is a representation embedding $ C_\phi \otimes B $ is indecomposable and hence (\cite[1.2.32]{PreNBK}) $ \beta \phi  $ is $ +$-irreducible in the lattice $ {\rm pp}^n_R$, so $ \beta \phi =\beta \psi _i $ for some $ i $.  Replacing our original choice of $\psi$ by this $\psi_i$ we may, therefore, without loss of generality, suppose that $ C_\psi  $ also is indecomposable.  Therefore $ f\otimes 1_B $ is actually an isomorphism. So, since $-\otimes B$ is a representation embedding and therefore exact and faithful, $f$ is an isomorphism - contradicting $\phi > \psi$.

For the general case, decompose $ C_\phi  $ and hence $ \phi $ as above: say $ \phi =\phi _1+\dots+\phi _p $ with each $ \phi _i $ $+$-irreducible. Some $ \phi _i$, say $ \phi _1$, is not contained in $ \psi  $, so we have $ \phi _1>\phi _1\wedge \psi $. By the case just dealt with we must have $ \beta \phi _1>\beta (\phi _1\wedge \psi )=\beta \phi _1\wedge \beta \psi $. By modularity $ \beta \phi _1+\beta \psi >\beta \psi $, but $ \beta \psi =\beta \phi \geq \beta \phi _1+\beta \psi  $ - a contradiction as required.
\end{proof}

\begin{cor}\label{KGfd} Suppose that $R$ and $S$ are Krull-Schmidt rings and there is a representation embedding from ${\rm Mod}\mbox{-}S$ to ${\rm Mod}\mbox{-}R$.  If the Krull-Gabriel dimension of $R$ is defined, then ${\rm KG}(S)\leq {\rm KG}(R)$.  Similarly for width and uniserial dimension (for these dimensions, see \cite[\S\S 7, 13]{PreNBK}).
\end{cor}

For example, the categories of modules over the domestic algebras $\Lambda_n$ considered in \cite{PreBur} and \cite{SchrKGdim} form a strictly increasing chain with respect to ``containment" {\it via} representation embeddings.  There are obvious representation embeddings from ${\rm mod}\mbox{-}\Lambda_m$ to ${\rm mod}\mbox{-}\Lambda_n$ if $m\leq n$ but the fact that ${\rm KG}(\Lambda_n) =n+1$ precludes there being representation embeddings in the the other direction.

\begin{cor} Suppose that $R$ and $S$ are Krull-Schmidt rings and there is a representation embedding from ${\rm Mod}\mbox{-}S$ to ${\rm Mod}\mbox{-}R$.  If the width of $S$ is undefined then so is that of $R$.
\end{cor}

Since the width of the lattice of pp formulas for the archetypal wild algebra $k\langle X, Y\rangle$, as well as its finite-dimensional avatars such as $k\xymatrix{\bullet \ar@/^/[r] \ar[r] \ar@/_/[r] & \bullet}$, is undefined (e.g.~\cite[7.3.27]{PreNBK}), we get the next result, where the second statement follows from a result of Ziegler (\cite[7.8(2)]{Zie}).

\begin{cor} If $R$ is a finite-dimensional algebra of wild representation type then the width of the lattice of pp formulas for $R$-modules is undefined.  In particular if $R$ is countable then there will be a superdecomposable pure-injective module.
\end{cor}

\section{Interpretation functors}\label{secinterp}\marginpar{secinterp}

The notion of interpretation comes from model theory.  Roughly, a structure interprets another if the first ``definably contains" the second.  This is made precise by using the notion of definable sorts (or ``imaginaries").  In the additive context we usually want the additive structure to be preserved by any interpretation, and this means that the formulas used to specify the interpretation should be pp formulas.  This also has the result that an interpretation is functorial.  It has been proved (\cite[7.2]{KrExactDef}, \cite[12.9]{PreMAMS}) that the interpretation functors between definable categories of modules are exactly those which preserve direct products and direct limits - examples are representation embeddings, representable functors $(A,-)$ when $A$ is a finitely presented module, ${\rm Ext}^n(A,-)$ when $A$ is FP$_{n+1}$ (\cite[\S 4]{PreInterp}).  But we also need the original definition from \cite{PreInterp}, repeated below\footnote{though it is not the most general form since definable categories may be definable subcategories of rings with many objects; for the general form see \cite[Chpt.~25]{PreMAMS}}, not least because this is often how interpretation functors are, in practice, specified.

An {\bf interpretation functor} $I:\mathcal{D}\rightarrow {\rm Mod}\mbox{-}S$, where ${\cal D}$ is a definable category is specified (up to equivalence) by giving a pp-$m$-pair $\phi/\psi$ - that is a pair of pp formulas each with $m$ free variables and with $\phi \geq \psi$ - and, for each $s\in S$, a pp-$2m$-formula (one with $2m$ free variables) $\rho_s$ such that, for all $M$ in the definable subcategory $\mathcal{D}$, the solution set $\rho_s(M,M)\subseteq M^m\times M^m$ defines an endomorphism of $\phi(M)/\psi(M)$ as an abelian group, and such that, together, these definable actions of elements of $S$ give $\phi(M)/\psi(M)$ the structure of an $S$-module with the action of $s\in S$ being given by $\rho_s$ (see \cite{PreInterp} or \cite[18.2.1]{PreNBK}).

In the following remark, we spell out the conditions on a pp-$2m$-formula $\rho$ which ensure that it defines an endomorphism of the solution set, $\phi(M)/\psi(M)$, of a pp-$m$-pair $\phi/\psi$ on a module $M$.  The intention is that $\rho$ should define an endomorphism of $\phi(M)/\psi(M)$ by sending $\overline{a} +\psi(M)$ with $\overline{a}\in\phi(M)$ to $\overline{b}+\psi(M)$, where $\overline{b}\in\phi(M)$ with $(\overline{a}, \overline{b}) \in \rho(M)$. Condition (1) means that for every $\overline{a}\in \phi(M)$ there is such an $\overline{b}\in\phi(M)$ and condition (2) ensures that if $\overline{a}\in\psi(M)$ then $\overline{b}\in\psi(M)$. Hence $\rho$ defines a well-defined map (necessarily additive since $\rho$ is pp).

\begin{remark}
If $\phi/\psi$ is a pp-$m$-pair, $\rho$ is a pp-$2m$-pair and $M$ is a module then the solution set $\rho(M,M)\subseteq M^m\times M^m$ defines an endomorphism of $\phi(M)/\psi(M)$ if and only if:

(1) for every $\overline{a}\in \phi(M)$ there is $\overline{b} \in \phi(M)$ such that $(\overline{a}, \overline{b})\in \rho(M)$;

(2) if $\overline{a}\in \psi(M)$ and $(\overline{a}, \overline{b}) \in \rho(M)$ then $\overline{b} \in \psi(M)$.

Each condition may be expressed by saying that an implication between pp formulas holds in $M$, namely:

(1) $\phi(\overline{x})\rightarrow \exists \overline{y} [\phi(\overline{y})\wedge \rho(\overline{x}, \overline{y})]$ and

(2) $\exists \overline{x} [\psi(\overline{x})\wedge \rho(\overline{x}, \overline{y})] \rightarrow \psi(\overline{y})$.

So to say that these hold for every $M$ in a definable category ${\cal D}$ is equivalent to the condition that the corresponding pp-pairs be closed on ${\cal D}$, where a pp-pair $\phi/\psi$ is {\bf closed} on $M$ if $\phi(M)=\psi(M)$ (we also say that $\phi/\psi$ is {\bf closed} on ${\cal D}$ if it is closed on every $M\in {\cal D}$; definable subcategories are exactly those obtained by specifying that some set of pp-pair be closed \cite[3.4.7]{PreNBK}).
\end{remark}

There is the following algebraic characterisation of interpretation functors.

\begin{theorem} (\cite[25.3]{PreMAMS}) An (additive as always) functor $I:{\cal C} \rightarrow {\cal D}$ between definable categories is an interpretation functor iff it commutes with direct products and direct limits.
\end{theorem}

In particular representation embeddings are examples; more generally, since tensoring with a finitely presented module preserves both direct products (\cite[S\"{a}tze 1,2]{Len}) and direct limits we have the following.

\begin{cor}\label{repembinterp}\marginpar{repembinterp} If $_SB_R$ is a bimodule with $_SB$ finitely presented then $-\otimes_SB_R:{\rm Mod}\mbox{-}S \rightarrow {\rm Mod}\mbox{-}R$ is an interpretation functor.
\end{cor}

Note that when $R$ and $S$ are both $k$-algebras and $s_1,\dots ,s_n$ is a basis for $S$ over $k$, then we need only define $\rho_{s_1},\dots ,\rho_{s_n}$ and extend $k$-linearly.  Indeed, in order to specify an interpretation functor it is enough that the actions of a generating set of $S$ (as a $k$-algebra) be specified and, if $S$ is finitely presented as an algebra, then only finitely many conditions (along the lines of those seen in the next proof) need be added to ensure that we have an $S$-action.

\begin{definition} The {\bf Ziegler spectrum}, ${\rm Zg}_R$ of a ring $R$ is the topological space which has, for its points, the isomorphism classes of indecomposable pure-injective (right) $R$-modules and with a basis of open sets being the sets:
\[(\phi/\psi) = \{ N\in {\rm Zg}_R: \phi(N)/\psi(N)\neq 0\}\]
where $\phi>\psi$ ranges over pairs of pp formulas; equivalently this basis consists of the sets:
\[(F) =\{ N\in {\rm Zg}_R: FN\neq 0\}\]
for $F:{\rm Mod}\mbox{-}R\rightarrow {\rm Ab}$ a coherent functor.

We use the fact that each of these basic open sets is compact (in the sense that every open covering is finite).  Furthermore, if $R$ is a finite-dimensional algebra, then each indecomposable finite-dimensional $R$-module is pure-injective, hence a point of ${\rm Zg}_R$, indeed an isolated point and, together, the finite-dimensional points are dense in ${\rm Zg}_R$ (see the background references, for instance \cite[Section 5.3]{PreNBK} for all this).
\end{definition}

\begin{definition}
We say that a definable subcategory $\mathcal{D}\subseteq {\rm Mod}\mbox{-}R$ is {\bf finitely axiomatisable} if there exist finitely many pp-pairs $\phi_1/\psi_1,\dots ,\phi_t/\psi_t$ such that
\[\mathcal{D}=\{M\in{\rm Mod}\mbox{-}R  \ \vert \  \phi_i(M) =\psi_i(M) \text{ for } 1\leq i\leq t\}.\]
Equivalently, a definable subcategory is finitely axiomatisable if the corresponding closed subset of the Ziegler spectrum is the complement of a compact open subset \cite[4.9]{Zie}.
\end{definition}

In connection with proving undecidability results, given an interpretation functor $I:\mathcal{D}\rightarrow {\rm Mod}\mbox{-}R$ it is useful to be able to replace $\mathcal{D}$ by a finitely axiomatisable subcategory $\mathcal{D}'\supseteq\mathcal{D}$ and $I$ by an interpretation functor $I'$ which extends $I$.  For then, if the theory of ${\cal D}'$ interprets the word problem for groups, so does the theory of ${\rm Mod}\mbox{-}R$ (since the conditions for membership of ${\cal D}'$ can be said as a single sentence).

\begin{proposition}\label{extendtofinax}
Let $S,R$ be $k$-algebras with $S$ finite-dimensional, more generally \cite[2.6.4]{CohNew}, with $S$ finitely presented as a $k$-algebra. Suppose that $\mathcal{D}\subseteq {\rm Mod}\mbox{-}R$ is a definable subcategory and that $I:\mathcal{D}\rightarrow {\rm Mod}\mbox{-}S$ is an interpretation functor. Then $I$ may be extended to an interpretation functor from $\mathcal{D}'$ to ${\rm Mod}\mbox{-}S$ for some finitely axiomatisable definable subcategory $\mathcal{D}'\supseteq \mathcal{D}$ of ${\rm Mod}\mbox{-}R$.
\end{proposition}
\begin{proof}
Suppose that data defining $I$ (in the sense above) are:  a pp-$m$-pair $\phi/\psi$  which gives the object part of $I$;  pp-$2m$-formulas $\rho_{s_1}, \dots, \rho_{s_n}$ which define the actions of a chosen finite generating set $s_1, \dots, s_n$ of $S$ as a $k$-algebra.  We will show how to write down finitely many conditions which express that this data define an interpretation functor.  Each condition will be an implication between pp formulas, that is, closure of a pp-pair.  So, together, these will cut out a finitely axiomatisable subcategory ${\cal D}'$ of ${\rm Mod}\mbox{-}R$ which contains ${\cal D}$ and on which this data defines a functor $I'$ extending $I$ as required.

The conditions that each $\rho_{s_i}$ defines a well-defined function on $\phi(-)/\psi(-)$ have already been observed, above, to be implications of pp-pairs, so it remains to say that, together, they do define an action of $S$ on $\phi(-)/\psi(-)$.  The point is that each relation between the $s_i$ is the condition that a certain noncommutative polynomial $p$ over $k$ is such that $p(s_1,\dots, s_n)=0$, and the translation of this to the same condition on the actions of the $\rho_{s_i}$ on $\phi$ modulo $\psi$ is expressible by an implication between pp formulas.  We illustrate this with the case where the $s_i$ form a $k$-basis for $S$; the general case uses the same ideas but is messy to write down.

So suppose that the $s_i$ form a basis for $S$.  For $1\leq i,j\leq n$, let $\alpha_{ij}^1,\dots ,\alpha_{ij}^n\in k$ be such that \[s_is_j= \alpha_{ij}^1s_1+...+\alpha_{ij}^ns_n.\]
Then an implication of pp-$m$-formulas which expresses that the composition of the action defined by $\rho_i$ followed by that of $\rho_j$ does indeed give the correct linear composition of actions of $\rho_{s_1}, \dots, \rho_{s_n}$ is:

\[\phi(\overline{x}) \rightarrow \exists \overline{u}, \overline{v}, \overline{w}_1,\dots ,\overline{w}_n \ \rho_{s_i}(\overline{x},\overline{u}) \wedge \rho_{s_j}(\overline{u},\overline{v}) \wedge \bigwedge_{l=1}^n \rho_{s_l}(\overline{x},\overline{w}_l) \wedge \, \psi(\overline{v} -\sum_{l=1}^n \alpha_{ij}^l \overline{w}_l).\]

If we take ${\cal D}'$ to be the finitely axiomatised subcategory of ${\rm Mod}\mbox{-}R$ defined by the finitely many pp-pairs whose closures are exactly these implications and those mentioned earlier, then the data $(\phi/\psi ;\rho_{s_1},\dots ,\rho_{s_n})$ defines an interpretation functor $I'$ from $\mathcal{D}'$ to ${\rm Mod}\mbox{-}S$ which extends $I$.
\end{proof}

Associated to each definable category ${\cal D}$ there is a skeletally small abelian category ${\rm fun}({\cal D})$ which may be obtained as the category of functors from ${\cal D}$ to ${\bf Ab}$ which commute with direct products and direct limits and which, alternatively, is obtained as the category, ${\mathbb L}^{\rm eq+}({\cal D})$, of pp-sorts on ${\cal D}$.  The latter has, for its objects, the sorts $\phi/\psi$ corresponding to pp-pairs $\psi\leq \phi$ and, for its morphisms, the pp-definable-with-respect-to-${\cal D}$ maps between sorts.  In the case that ${\cal D}$ is a module category ${\rm Mod}\mbox{-}R$, we write ${\rm fun}\mbox{-}R ={\rm fun}({\rm Mod}\mbox{-}R)$ and have that ${\rm fun}\mbox{-}R = ({\rm mod}\mbox{-}R, {\bf Ab})^{\rm fp}$ is equivalent to the category of coherent functors on ${\rm Mod}\mbox{-}R$; since every finitely presented functor on ${\rm mod}\mbox{-}R$ is a subquotient of a power of the forgetful functor one sees why the objects also correspond to pp-pairs.  We also set ${\rm Fun}\mbox{-}R = ({\rm mod}\mbox{-}R, {\bf Ab})$ to be the locally coherent Grothendieck abelian category which has ${\rm fun}\mbox{-}R$ for its category of finitely presented objects.  If ${\cal D}$ is a definable subcategory of ${\rm Mod}\mbox{-}R$ then ${\rm fun}({\cal D})$ (respectively ${\rm Fun}({\cal D})$) is the quotient category of ${\rm fun}\mbox{-}R$ (resp.~${\rm Fun}\mbox{-}R$) by the Serre (resp.~torsion) subcategory ${\cal S}_{\cal D}$ of those functors which are $0$ on ${\cal D}$ (that is, those corresponding to pp-pairs which are closed on ${\cal D}$).  Furthermore the relation between ${\rm fun}({\cal D})$ and ${\rm Fun}({\cal D})$ is just as in the case ${\cal D} = {\rm Mod}\mbox{-}R$.  For the equivalence between the category $({\rm mod}\mbox{-}R, {\bf Ab})^{\rm fp}$ of finitely presented functors on finitely presented modules and the category ${\mathbb L}^{\rm eq+}_R = {\mathbb L}^{\rm eq+}({\cal D})$ of pp-sorts (also called pp-imaginaries) see \cite[10.2.30]{PreNBK} and, for further detail about this background material, \cite{PreNBK}, also \cite{PreMAMS} and \cite{PreAbex}.

To any interpretation functor $I:{\cal C} \rightarrow {\cal D}$ we associate the functor $I_0:{\rm fun}({\cal D}) \rightarrow {\rm fun}({\cal C})$ which is given on objects by sending $B\in {\rm fun}({\cal D})$ to the composition $B I$ and by sending a natural transformation $\tau:B_1\rightarrow B_2$ to the natural transformation whose component at $C\in {\cal C}$ is given by $\tau_{IC}$.  We use the basic formula  $$I_0G.C = G.IC \hspace{12pt} (\ast)$$ for $C\in {\cal C}$ and $G\in {\rm fun}({\cal D})$.  See \cite[Chpt.~13]{PreMAMS} (or \cite[Chpt.~18]{PreNBK}) for more detail, including the facts that this is exact and that all exact functors from ${\rm fun}\mbox{-}S$ to ${\rm fun}\mbox{-}R$ arise in this way.  Indeed (\cite{PreRajStrShv}, \cite{PreAbex}) this gives an equivalence between the 2-categories of definable categories and small abelian categories.

\begin{lemma}\label{allemb}
Let $I:{\rm Mod}\mbox{-}R\rightarrow {\rm Mod}\mbox{-}S$ be an interpretation functor. The definable subcategory generated by the image of $I$ is ${\rm Mod}\mbox{-}S$ if and only if $I_0$ is faithful.
\end{lemma}
\begin{proof}
This follows from the correspondence between definable subcategories of ${\rm Mod}\mbox{-}S$ and Serre subcategories of (that is, kernels of exact functors from) ${\rm fun}\mbox{-}S$, see \cite[12.4.1]{PreNBK}.  Since $I_0$ is exact, it is faithful if and only if for all non-zero $G \in {\rm fun}\mbox{-}S$, $I_0G\neq 0$. By ($\ast$) above, this is true if and only if for all $G\in{\rm fun}\mbox{-}S$, there exists an $M\in {\rm Mod}\mbox{-}R$ such that $G IM\neq 0$. So $I_0$ is faithful if and only if for every proper definable subcategory $\mathcal{D}$ of ${\rm Mod}\mbox{-}S$, there exists an $M\in{\rm Mod}\mbox{-}R$ with $IM\notin\mathcal{D}$.
\end{proof}

In the remainder of this section we investigate lattice homomorphisms between pp-lattices induced by interpretation functors.

Let ${\cal A}$ be an abelian category and $A$ an object of ${\cal A}$. Recall that a subobject of $A$ is an equivalence class of monomorphisms $i:X\to A$, where $i$ is equivalent to $j:Y\to A$ if there is an isomorphism $g:X\rightarrow Y$ with $i=jg$.
The collection of subobjects $\text{Sub}(A)$ of $A$ is ordered by defining $i\leq j$ if there is $k$ with $i=jk$.
It has meets given by pullbacks and joins given by pushouts.

We note the following.

\begin{lemma}\label{sublat} Let $F:{\cal A}\rightarrow {\cal B}$ be an exact functor. For each $A\in{\cal A}$, the map $\overline{F_{A}}:\text{Sub}(A)\rightarrow\text{Sub}(FA)$, induced by sending a monomorphism $i$ to $Fi$, is well-defined, order preserving and preserves both meet and join. Moreover, $F$ is faithful if and only if $\overline{F_{A}}$ is an embedding for all $A\in {\cal A}$.
\end{lemma}
\begin{proof}
Since $F$ is exact it maps subobjects of $A$ to subobjects of $FA$ and preserves pullbacks and pushouts (the operations which define meet and join respectively).

If $F$ is faithful and $i:X\rightarrow A, \ j:Y\rightarrow A$ are comparable subobjects - say there is a monomorphism $g:X\rightarrow Y$ with $i=jg$ - but $Fi$ is equal to $Fj$ as a subobject of $FA$, then $F$ must annihilate the cokernel of $g$. So, since $F$ is faithful, that cokernel is $0$. Hence $i$ and $j$ are equal as subobjects of $A$.  The converse is clear since, $F$ being exact, in order to check that $F$ is faithful it is enough to show that if $FA=0$ then $A=0$.
\end{proof}

Now suppose that $\psi\leq \phi$ is a pp-pair for $R$-modules.  Then (see the comments at the beginning of Section \ref{seclat}) the subobject lattice of the corresponding object $\phi/\psi$ in ${\mathbb L}^{\rm eq+}_R$($\simeq {\rm fun}\mbox{-}R$) is naturally isomorphic to the interval, $[\phi,\psi]$, between $\psi$ and $\phi$ in the pp-lattice ${\rm pp}_R^n$.

\begin{cor}\label{intfunctorslatthom} Let $R$ and $S$ be rings.
Suppose that $I:{\rm Mod}\mbox{-}R\rightarrow {\rm Mod}\mbox{-}S$ is an interpretation functor with $IM= \phi(M)/\psi(M)$ (as abelian groups) for $M\in {\rm Mod}\mbox{-}R$.  Then there is an induced lattice homomorphism from ${\rm pp}_S^1$ to the interval $[\phi,\psi]$ in ${\rm pp}_R^n$.

In the case that the smallest definable subcategory of ${\rm Mod}\mbox{-}S$ containing the image of $I$ is the whole of ${\rm Mod}\mbox{-}S$, equivalently that $I_0$ is faithful, this is an embedding.
\end{cor}
\begin{proof}
The exact functor $I_0:{\mathbb L}^{\rm eq+}_S \rightarrow{\mathbb L}^{\rm eq+}_R$ induced by $I$ sends the pp-pair $x=x/x=0$ to $\phi/\psi$. By \ref{sublat}, $I_0$ induces a lattice homomorphism from the subobject lattice of $x=x/x=0$ to the subobject lattice of $\phi/\psi$.  The second statement follows by \ref{sublat} and \ref{allemb}.
\end{proof}

\begin{proposition}\label{Nag}\cite[Proposition 2.3]{Nag}
Let $R,S$ be finite-dimensional $k$-algebras.
Let $F:=(-\otimes_S B_R):{\rm mod}\mbox{-}S\rightarrow {\rm mod}\mbox{-}R$ be a representation embedding and let $G:=(B_R,-) :{\rm mod}\mbox{-}R\rightarrow {\rm mod}\mbox{-}S$ be its right adjoint. For every $N\in {\rm mod}\mbox{-}S$, $N$ is a direct summand of $GFN$.
\end{proposition}

If $S$ is a finite-dimensional algebra then, for every proper definable subcategory $\mathcal{D}$, there is a finite-dimensional $S$-module $N$ such that $N\notin\mathcal{D}$ (because the set of finite-dimensional indecomposables is dense in the Ziegler spectrum of $S$, see \cite[5.3.36]{PreNBK}). So \ref{allemb} implies that, in the situation of the above proposition, the definable subcategory generated by the image of $(B_R,-):{\rm Mod}\mbox{-}R\rightarrow {\rm Mod}\mbox{-}S$ is the whole of ${\rm Mod}\mbox{-}S$ (in the finitely controlled case, see Section \ref{seccontr}, this will also follow from \ref{controlled restriction of scalars}).

\begin{cor}\label{lattagain}
Let $R,S$ be finite-dimensional $k$-algebras.  Suppose that $F:=(-\otimes \,_SB_R):{\rm mod}\mbox{-}S\rightarrow {\rm mod}\mbox{-}R$ is a representation embedding and set $I$ to be the interpretation functor $(_SB_R,-) :{\rm Mod}\mbox{-}R\rightarrow {\rm Mod}\mbox{-}S$.  Then the corresponding functor $I_0:{\mathbb L}^{\rm eq+}_S \rightarrow {\mathbb L}^{\rm eq+}_R$ induces a lattice embedding from ${\rm pp}_S^1$ to an interval $[\psi(x), x=0]$ in the lattice ${\rm pp}_R^n$ where $\psi$ is some quantifier-free pp formula.
\end{cor}
\begin{proof}
As noted above, $I_0$ is faithful.  It sends $(S,-)$, that is, $x=x/x=0 \in{\mathbb L}^{\rm eq+}_S$, to $(S,I(-))\in{\mathbb L}^{\rm eq+}_R$ and $(S,I(-))\cong(FS,-)=(B_R,-)$.  Representable functors in ${\mathbb L}^{\rm eq+}_R$ correspond to pp-pairs of the form $\psi(x)/x=0$ where $\psi(x)$ is quantifier-free:  in particular if $\overline{t}:R^n\rightarrow B_R$ is epi then $(\overline{t},-): (B_R,-) \rightarrow (R^n,-)$ is an embedding with image of the form $F_\psi$ with $\psi$ quantifier-free (\cite[10.2.34]{PreNBK}). So, by \ref{intfunctorslatthom}, $I_0$ induces a lattice embedding of $\text{pp}_1^S = [x=x, x=0]$ into the interval $[\psi(x), x=0]$ in $\text{pp}_n^R$.
\end{proof}

This gave us another route to \ref{latt} (at least for finite-dimensional algebras).  In fact, as we now show, the explicit construction of a lattice homomorphism $\beta$ from $\text{pp}_S^1$ to $\text{pp}_R^n$ in Section \ref{seclat} is essentially the same as the lattice homomorphism defined above by the interpretation functor $(_SB_R,-):{\rm Mod}\mbox{-}S\rightarrow {\rm Mod}\mbox{-}R$.

First note that both constructions are dependent upon picking an $n$-tuple $\overline{t}$ generating $B_R$. This is clear in Section \ref{seclat}. In this section, such a choice is implicit in the proof of \ref{lattagain}, where we identify ${\rm Hom}_R(_SB_R,-)$ composed with the forgetful functor from ${\rm Mod}\mbox{-}S$ to ${\rm Ab}$ with a pp-pair.  Indeed, if $\overline{t}$ generates $B_R$ then this induces an embedding $(\overline{t},-):{\rm Hom}_R(B,-)\rightarrow {\rm Hom}_R(R^n,-)$, the image of which, as a subfunctor of ${\rm Hom}_R(R^n,-)$, is given by a (quantifier-free) pp formula $\psi$ which generates the pp-type of $\overline{t}$ in $B_R$.

Now let $(C,c)$ be a free realisation of $\phi\in\text{pp}_S^1$. So $\text{im}(c,-)=F_\phi\subseteq (S,-)$. Set $I= {\rm Hom}_R(_SB_R,-)$; since $I_0:{\rm fun}\mbox{-}S \rightarrow {\rm fun}\mbox{-}R$ is exact, $I_0$ sends $\text{im}(c,-)$ to the image of $I_0(c,-)$. Using ($\ast$) and the Hom-tensor adjunction, we have that $I_0$ applied to ${\rm Hom}_S(C,-)$ is ${\rm Hom}_R(C\otimes_SB_R,-)$ and, moreover that $I_0(c,-) = (c\otimes 1_{B_R},-):(C\otimes_SB_R,-)\rightarrow (S\otimes_SB_R,-)$. So the construction in \ref{lattagain}, sends $F_\phi$ to $\text{im}(c\otimes\overline{t},-)\subseteq (R^n,-)$. That is, in terms of formulas, $\phi$ is sent to a generator, i.e.~$\beta\phi$, of the pp-type of $c\otimes\overline{t}$ in $C\otimes_SB_R$.

\begin{remark}
Corollary \ref{intfunctorslatthom} allows us to extend \ref{KGfd}:  if $I:{\rm Mod}\mbox{-}R\rightarrow {\rm Mod}\mbox{-}S$ is an interpretation functor such that the smallest definable subcategory of ${\rm Mod}\mbox{-}S$ containing the image of $I$ is the whole of ${\rm Mod}\mbox{-}S$ then $\text{KG}(R)\leq \text{KG}(S)$.
\end{remark}

%

\section{Controlled wild algebras}\label{seccontr}\marginpar{seccontr}

Let $S$ be a $k$-algebra, not necessarily finite-dimensional. Denote by ${\rm fin}\mbox{-}S$ the category of finite-dimensional right $S$-modules. Let $R$ be a finite-dimensional $k$-algebra. A faithful exact functor $F: {\rm fin}\mbox{-}S\rightarrow {\rm mod}\mbox{-}R$ is a \textbf{controlled representation embedding} if there exists a full subcategory $\mathcal{C}$ of ${\rm mod}\mbox{-}R$, closed under direct sums and direct summands, such that, for all $M,N\in{\rm fin}\mbox{-}S$,
\[{\rm Hom}_R(FM,FN)=F{\rm Hom}_{S}(M,N)\oplus{\rm Hom}_R(FM,FN)_{\mathcal{C}}\] and
\[{\rm Hom}_R(FM,FN)_{\mathcal{C}}\subseteq{\rm rad}({\rm Hom}_R(FM,FN))\] where ${\rm Hom}_R(FM,FN)_{\mathcal{C}}$ denotes the set of morphisms from $FM$ to $FN$ which factor through some $C\in \mathcal{C}$. We say that $F$ is \textbf{controlled by} $\mathcal{C}$, or ${\cal C}$-controlled. We say that $F$ is \textbf{finitely controlled} if it is controlled by $\text{add}(C)$ for some $C\in {\rm mod}\mbox{-}R$. That a controlled representation embedding $F:{\rm mod}\mbox{-}S\rightarrow {\rm mod}\mbox{-}R$, where $S,R$ are finite-dimensional algebras, is a representation embedding can be extracted directly from the proof of \cite[Proposition 2.2]{Han}.

We say that a finite-dimensional $k$-algebra $R$ is ({\bf finitely}) {\bf controlled wild} if there is a (finitely) controlled representation embedding $F:{\rm fin}\mbox{-}k\langle X,Y\rangle\rightarrow {\rm mod}\mbox{-}R$.

\begin{remark}\label{factoringcontrolledrepembeddings}
We recall a simplification from \cite{PreRepn}.  If $_SB_R$ is a bimodule such that $_SB$ is a finitely generated projective generator of ${\rm Mod}\mbox{-}S$ then the functor $F=(-\otimes _SB_R):{\rm Mod}\mbox{-}S\rightarrow {\rm Mod}\mbox{-}R$ can be factored as a Morita equivalence followed by restriction of scalars.  Namely, let $T={\rm End}({_S}B)$; so the right action of $R$ on $B$ induces a homomorphism $\lambda:R\rightarrow T$ and $F=GH$ where $H=(-\otimes_SB_T): {\rm Mod}\mbox{-}S\rightarrow {\rm Mod}\mbox{-}T$ is the Morita equivalence and $G:{\rm Mod}\mbox{-}T\rightarrow {\rm Mod}\mbox{-}R$ is the restriction of scalars map induced by $\lambda$.  It is immediate that $G$ is a representation embedding iff $F$ is.  Also, if $F$ is controlled by ${\cal C}$ then so is $G$, since $H$ induces an equivalence between ${\rm fin}\mbox{-}S$ and ${\rm fin}\mbox{-}T$.

Therefore we may assume that $B$ is $_SS_R$ and the functor $F$ is just restriction of scalars from $S$ to a subring $R$, in which case the first formula above becomes:
\[{\rm Hom}_R(M,N)={\rm Hom}_{S}(M,N)\oplus{\rm Hom}_R(M,N)_{\mathcal{C}}\]
for all $M,N\in {\rm fin}\mbox{-}S$.  Indeed, provided $M$ is finite-dimensional, this is true for any $N\in {\rm Mod}\mbox{-}S$ (use that any morphism from $M$ factors through a finite-dimensional submodule of $N$). In particular it holds for $M=S$.
\end{remark}

\begin{lemma}\label{factormap}
Let $R$ be a finite-dimensional $k$-algebra. Suppose $M,C\in{\rm mod}\mbox{-}R$. There exists $n\in{\mathbb N}$ and $\Delta\in {\rm Hom}_R(M,C^n)$ such that any morphism from $M$ to any $N\in {\rm add}(C)$ factors initially through $\Delta$.  That is, $\Delta$ is an ${\rm add}(C)$-preenvelope of $M$.
\end{lemma}

\begin{proof}
Since ${\rm Hom}_R(M,C)$ is finite-dimensional there are $f_1,\dots ,f_n$ which span ${\rm Hom}_R(M,C)$ as a $k$-vector space. Let $\Delta:M\rightarrow C^n$ be given coordinatewise by (the transpose of) $(f_1,\dots ,f_n)$. It will be enough to consider morphisms from $M$ to powers of $C$. Suppose $g:M\rightarrow C^m$ is given coordinatewise as $g=(g_1,\dots, g_m)^T$.  For each $i=1,\dots,m$, let $\lambda_{ij}\in k$ be such that $g_i=\sum_{j=1}^n \lambda_{ij}f_j$. Define the map $h:C^m\rightarrow C^n$ to be that given by the matrix $(\lambda_{ij})_{ij}$.  Then $h\Delta = g$, as required.
\end{proof}

\begin{theorem}\label{controlled restriction of scalars}
Suppose $S,R$ are finite-dimensional $k$-algebras and $F:{\rm Mod}\mbox{-}S\rightarrow {\rm Mod}\mbox{-}R$ is such that its restriction to ${\rm mod}\mbox{-}S$ is a finitely controlled representation embedding, controlled by $\mathcal{C} \subseteq {\rm mod}\mbox{-}R$. Then there is an interpretation functor $I:{\rm Mod}\mbox{-}R\rightarrow {\rm Mod}\mbox{-}S$ such that $IFM\cong M$ for all $S$-modules $M$.

More generally, we have this for a ${\cal C}$-controlled representation embedding, $-\otimes _SB_R$, provided $B_R$ has a ${\cal C}$-preenvelope in ${\rm mod}\mbox{-}R$.
\end{theorem}
\begin{proof}  In order to define $I$, we must find a pp-pair $\phi/\psi$ of pp formulas such that, given an $R$-module $N$, the quotient $\phi(N)/\psi(N)$ will be the underlying group of $IN$ and, for each element $s\in S$, we have to show that multiplication by $s$ on that underlying group can be defined by a pp formula.

First note that it is enough to consider the case where $F$ is restriction of scalars from a ring $S$ to a subring $R$: use Remark \ref{factoringcontrolledrepembeddings} combined with the easily checked facts that Morita equivalences are interpretation functors (originally from \cite[1.1]{PoPr}), as is the inclusion of ${\rm Mod}\mbox{-}R'$ into ${\rm Mod}\mbox{-}R$ whenever $R\rightarrow R'$ is a surjection of rings.

Next, for every $R$-module $N$, if $s\in S$ and $f:S_R\rightarrow N$ factors through $\mathcal{C}$ then so does $f\cdot s$ where $f\cdot s$ means left multiplication by $s$ on $S$ followed by $f$ - the right $S$-module structure on ${\rm Hom}_R(S,N)$.  Thus the vector space decomposition $${\rm Hom}_R(S,N)={\rm Hom}_S(S,N)\oplus{\rm Hom}_R(S,N)_\mathcal{C} \hspace{4pt} (\ast\ast)$$ is a decomposition into right $S$-modules.  And, in the case that $N=FM$ is the reduction of an $S$-module, the first factor is isomorphic to $M$.  Also, if $g:N\rightarrow M$ is a morphism of $R$-modules then the induced map ${\rm Hom}_R(S,g): {\rm Hom}_R(S,N) \rightarrow {\rm Hom}_R(S,M)$, which is given by postcomposition with $g$, is a morphism of right $S$-modules and carries  ${\rm Hom}_R(S,N)_{\cal C}$ to ${\rm Hom}_R(S,M)_\mathcal{C}$.  So ${\rm Hom}_R(S,-)_\mathcal{C}$ is a subfunctor of ${\rm Hom}_R(S,-):{\rm Mod}\mbox{-}R\rightarrow {\rm Mod}\mbox{-}S$.

Let $C\in {\rm mod}\mbox{-}R$ be such that ${\cal C}= \text{add}(C)$. Since $S_R$ is finite-dimensional, by \ref{factormap}, there exists $\Delta:S_R\rightarrow C^n$ such that for all $f:S_R\rightarrow N$ which factor through $\mathcal{C}$, there is $h:C^n\rightarrow N$ such that $f=h \Delta$. Thus the image of $(\Delta,-):{\rm Hom}_R(C^n,-) \rightarrow {\rm Hom}_R(S,-)$ is exactly ${\rm Hom}_R(S,-)_\mathcal{C}$.  The latter, therefore, is a finitely presented functor (precisely, the $\varinjlim$-commuting extension of a finitely presented functor on ${\rm mod}\mbox{-}R$ to one on ${\rm Mod}\mbox{-}R$) hence, \cite[10.2.43]{PreNBK}, is pp-definable.  Therefore, the quotient, ${\rm Hom}_R(S,-) / {\rm Hom}_R(S,-)_\mathcal{C}$ also is finitely presented, equivalently is given by a pp-pair.   Furthermore, each action of an element of $S$ on this quotient is, as we have just seen, an endomorphism in the functor category, hence \cite[Theorem 10.2.30]{PreNBK} is pp-definable (a `definable scalar' on the sort ${\rm Hom}_R(S,-) / {\rm Hom}_R(S,-)_\mathcal{C})$.  Thus we have an interpretation functor $I={\rm Hom}_R(S,-) / {\rm Hom}_R(S,-)_\mathcal{C}:{\rm Mod}\mbox{-}R \rightarrow {\rm Mod}\mbox{-}S$.

If we apply $I$ to the $R$-reduction of an $S$-module $M$ then we have, by ($\ast\ast$),  \[{\rm Hom}_R(S,M) / {\rm Hom}_R(S,M)_\mathcal{C} \simeq {\rm Hom}_S(S,M) \simeq M,\] that is, $IFM\simeq M$, which is what we want.
\end{proof}

Thus we have shown that if $M$ is an $S$-module then its image $N=FM$ under $F$ still contains a copy of $M$ in the model-theoretic structure $N^{\rm eq+}$ which is $N$ expanded by all the pp-definable sorts and pp-definable functions between these sorts (see \cite{KP1} or \cite[Appx.~B2]{PreNBK} for a definition of that structure).  That is, applying the representation embedding has hidden the original module but not lost it.

By \cite[Lemma 2.4]{Han}, a finitely controlled wild finite-dimensional algebra $R$ has a finitely controlled representation embedding from the modules over any strictly wild finite-dimensional algebra. So we deduce that, in this case, ${\rm Mod}\mbox{-}R$ interprets the module category over any strictly wild finite-dimensional algebra.

\begin{cor}
Let $R$ be a finitely controlled wild finite-dimensional $k$-algebra. For every strictly wild finite-dimensional $k$-algebra $S$, there is an interpretation functor $I:{\rm Mod}\mbox{-}R\rightarrow {\rm Mod}\mbox{-}S$ whose image is the whole of ${\rm Mod}\mbox{-}S$.
\end{cor}

Recall that $F= (-\otimes_S B_R)$ is said to be a {\bf strict representation embedding} if it is a full representation embedding, equivalently, if the induced $R\rightarrow {\rm End}(_SB)$ is a ring epimorphism; in this case, the image of $F$ is a definable subcategory of ${\rm Mod}\mbox{-}R$ (\cite[5.5.4]{PreNBK}) equivalent to ${\rm Mod}\mbox{-}S$.  It follows easily (e.g.~use results in \cite[\S XIX.1]{SiSk3}) that if $R$ is a strictly wild finite-dimensional algebra then, for every finitely generated algebra $T$, there is a definable subcategory of ${\rm Mod}\mbox{-}R$ equivalent to ${\rm Mod}\mbox{-}T$.

For finitely controlled wild algebras, the category of modules ``definably contains" the category of modules over any finitely generated algebra in the following weaker sense.

\begin{cor}
Let $R$ be a finitely controlled wild finite-dimensional $k$-algebra.  Then, for every finitely generated $k$-algebra $S$, there is an interpretation functor from some definable subcategory $\mathcal{C}$ of ${\rm Mod}\mbox{-}R$ to ${\rm Mod}\mbox{-}S$ whose image is the whole of ${\rm Mod}\mbox{-}S$.
\end{cor}
\begin{proof}
Let $T$ be a strictly wild finite-dimensional algebra.  Then, as observed above, there is a definable subcategory $\mathcal{D}$ of ${\rm Mod}\mbox{-}T$ which is equivalent to ${\rm Mod}\mbox{-}S$.  There is also a finitely-controlled representation embedding from ${\rm mod}\mbox{-}T$ to ${\rm mod}\mbox{-}R$ so, by \ref{controlled restriction of scalars} there is an interpretation functor $I:{\rm Mod}\mbox{-}R\rightarrow {\rm Mod}\mbox{-}T$ whose image is the whole of ${\rm Mod}\mbox{-}T$.  By \cite[13.3]{PreMAMS}, $I^{-1}\mathcal{D}$ is a definable subcategory of ${\rm Mod}\mbox{-}R$ and the restriction of $I$ to this definable subcategory is as required.
\end{proof}

In particular we can take $S=k\langle X,Y\rangle$.

\begin{cor}\label{fdtoinfd}
Let $R$ be a finitely controlled wild finite-dimensional $k$-algebra.  Then there is an interpretation functor from a definable subcategory $\mathcal{D}$ of ${\rm Mod}\mbox{-}R$ to ${\rm Mod}\mbox{-}k\langle X,Y\rangle$ whose image is all of ${\rm Mod}\mbox{-}k\langle X,Y\rangle$.
\end{cor}

\begin{cor}
Let $k$ be a countable recursively given field and $R$ a finitely controlled wild $k$-algebra. Then the theory of ${\rm Mod}\mbox{-}R$ is undecidable.
\end{cor}
\begin{proof}
The algebra $S=k\xymatrix{\bullet \ar@/^/[r] \ar[r] \ar@/_/[r] & \bullet}$ is finite-dimensional and has undecidable theory of modules (see \cite[17.3]{PreBk} for a proof and references).  So there can be no decision procedure for determining, given a sentence $\sigma$ in the theory of $S$-modules, whether there is an $S$-module which satisfies $\sigma$.  There is a finitely controlled representation embedding of $S$-modules into $R$-modules, so let $I$ be as in \ref{controlled restriction of scalars}, giving an interpretation of $S$-modules in $R$-modules.  Then, using $I$, there is a sentence $\sigma'$ in the language of $R$-modules such that an $R$-module $N$ satisfies $\sigma'$ iff $IN$ satisfies $\sigma$.  Therefore there can be no decision procedure for determining, given a sentence in the language of $R$-modules, whether there is an $R$-module satisfying that sentence, as required.
\end{proof}

Examples of finitely controlled wild algebras include, at least in the case that $k$ is algebraically closed, all local wild $k$-algebras and wild $k$-algebras with square zero radical (\cite[Proposition p.~290, Theorem p.~291]{Han}), so the conclusions above apply for all these.  We would like to extend the conclusions to all wild algebras; so far as we are aware, there are, currently, no known wild algebras which are not (even finitely) controlled wild.

\begin{remark}
The question arises, in \ref{fdtoinfd} above, whether there is an interpretation functor as there with the domain being all of ${\rm Mod}\mbox{-}R$.  We show that, in fact, the domain must be a proper definable subcategory:  if $R$ is a finite-dimensional $k$-algebra and $S$ is an infinite dimensional $k$-algebra then there is no interpretation functor from ${\rm Mod}\mbox{-}R$ to ${\rm Mod}\mbox{-}S$ with the definable subcategory generated by the image being all of ${\rm Mod}\mbox{-}S$.

Suppose that $F\in ({\rm mod}\mbox{-}R,\textbf{Ab})^{\text{fp}}$. Then there exists $M,N\in {\rm mod}\mbox{-}R$ and $f:M\rightarrow N$ such that $F$ is the cokernel of \[(f,-):(N,-)\rightarrow (M,-).\]  Since $(M,-)$ is projective, any endomorphism of $F$ lifts to one of $(M,-)$ and one sees that $\text{End}F$ is a subquotient of $\text{End} (M)^\text{op}$ and hence is a finite-dimensional $k$-algebra.

Suppose that $I:{\rm Mod}\mbox{-}R \rightarrow {\rm Mod}\mbox{-}S$ is an interpretation functor, with effect on objects being given by the coherent functor extending $F\in ({\rm mod}\mbox{-}R,\textbf{Ab})^{\text{fp}}$ and with the definable subcategory generated by the image being all of ${\rm Mod}\mbox{-}S$. Then, by \ref{allemb}, $I_0:{\mathbb L}^{\rm eq+}_S\rightarrow {\mathbb L}^{\rm eq+}_R$ is faithful and sends the forgetful functor in ${\mathbb L}^{\rm eq+}_S$ to $F$. Thus we have a ring embedding from $S$ to $\text{End} F$ and hence $S$ must be a finite-dimensional $k$-algebra.
\end{remark}

\section{Tame does not interpret wild}

In this section we show that, for finite-dimensional $k$-algebras $R,S$ where $k$ is algebraically closed, if $I:{\rm Mod}\mbox{-}R\rightarrow {\rm Mod}\mbox{-}S$ is an interpretation functor such that $\langle I{\rm Mod}\mbox{-}R \rangle={\rm Mod}\mbox{-}S$, then $S$ wild implies $R$ is wild.  Here we use $\langle {\cal X} \rangle $ to denote the definable subcategory generated by the modules in ${\cal X}$; when ${\cal X}$ is the image of an interpretation functor it is easy to see that this is the closure of ${\cal X}$ under pure subobjects (see \cite[3.8]{PreAbex}). Our result, together with \ref{Nag}, gives us the following new definition of a wild algebra in terms of interpretation functors.

\begin{theorem}
Let $k$ be algebraically closed and let $S:=k\xymatrix{\bullet \ar@/^/[r] \ar[r] \ar@/_/[r] & \bullet}$. A finite-dimensional $k$-algebra $R$ is wild if and only if there is an interpretation functor $I:{\rm Mod}\mbox{-}R\rightarrow {\rm Mod}\mbox{-}S$ such that $\langle I{\rm Mod}\mbox{-}R\rangle={\rm Mod}\mbox{-}S$.
\end{theorem}

Of course we could replace $k\xymatrix{\bullet \ar@/^/[r] \ar[r] \ar@/_/[r] & \bullet}$ in the above theorem by any other finite-dimensional wild algebra. 

To prove this, we show, \ref{boundingdimneededtohitfds}, that if $I:{\rm Mod}\mbox{-}R\rightarrow {\rm Mod}\mbox{-}S$ is an interpretation functor such that $\langle I{\rm Mod}\mbox{-}R\rangle={\rm Mod}\mbox{-}S$ then for each $d\in\mathbb{N}$ there is $n_d\in \mathbb{N}$ such that for all $M\in{\rm mod}\mbox{-}S$ of dimension less than or equal to $d$ there is $N\in{\rm mod}\mbox{-}R$ such that $M|IN$ and $\dim N\leq n_d$. A result of Krause, \cite[10.8]{KraSpec}, then shows that if $S$ is wild then $R$ is also wild.

Note that if $I:{\rm Mod}\mbox{-}R\rightarrow {\rm Mod}\mbox{-}S$ is an interpretation functor and $R$ is a finite-dimensional $k$-algebra then, since $I$ is given on objects by a subquotient of a finite power of the forgetful functor, for any finite-dimensional module $M\in {\rm Mod}\mbox{-}R$, $IM$ is finite-dimensional.

\begin{definition}
Let $R$ be a $k$-algebra. We write ${\rm ind}\mbox{-}R$ for the set of (isomorphism classes of) indecomposable finite-dimensional $R$-modules and for each $d\in\mathbb{N}$, ${\rm ind}_d\mbox{-}R$ will denote the set of those of dimension $d$.
\end{definition}

Suppose that $M\in{\rm ind}\mbox{-}R$, $a\in M$ is non-zero and $f:M\rightarrow N$ is left minimal almost split. Let $\phi$ generate the pp-type of $a$ in $M$ and let $\psi$ generate the pp-type of $f(a)$ in $N$.  Then, \cite[13.1]{PreBk}, $(\phi/\psi)$ isolates $M$ in $\text{Zg}_R$. We recall the argument here for the convenience of the reader. Suppose that $L\in{\rm mod}\mbox{-}R$ and $b\in\phi(L)$. There exists a map $g:M\rightarrow L$ such that $g(a)=b$. Now either $g$ is a section and hence $M$ is a direct summand of $L$ or there is a map $h:N\rightarrow L$ such that $hf=g$. So either $M$ is a direct summand of $L$ or $hf(a)=b$ and hence $b\in\psi(L)$. Thus, either the pp-type of $b$ is generated by $\phi$ or $b\in\psi(L)$. So for any pp-$1$-formula $\sigma<\phi$, we have $\sigma\leq \psi$. Now suppose that $L\in{\rm Mod}\mbox{-}R$ and $b\in\phi(L)$. Either $b\in\psi(L)$ or the pp-type of $b$ is generated by $\phi$. Thus, by \cite[4.3.48]{PreNBK}, $M$ is a direct summand of $L$.

So, for any module $M\in{\rm ind}\mbox{-}R$, there is a coherent functor $\overrightarrow{F_{\phi/\psi}}:{\rm Mod}\mbox{-}R\rightarrow \textbf{Ab}$ (with $\phi/\psi$ chosen as in the previous paragraph) such that for any $L\in{\rm Mod}\mbox{-}R$, $\overrightarrow{F_{\phi/\psi}}(L)\neq 0$ implies $M|L$.

\begin{proposition}\label{fdinimfd}
Let $R,S$ be finite-dimensional $k$-algebras. Let $\mathcal{D}$ be a definable subcategory of ${\rm Mod}\mbox{-}R$ which is generated as such by the finite-dimensional modules in it, meaning that $\mathcal{D}=\langle \mathcal{D}\cap{\rm mod}\mbox{-}R\rangle$. Let $I:\mathcal{D}\rightarrow {\rm Mod}\mbox{-}S$ be an interpretation functor such that $\langle I\mathcal{D}\rangle ={\rm Mod}\mbox{-}S$.
Then for every $M\in {\rm mod}\mbox{-}S$ there is $N\in\mathcal{D}\cap{\rm mod}\mbox{-}R$ such that $M$ is a direct summand of $IN$.
\end{proposition}
\begin{proof}
Let $M\in {\rm ind}\mbox{-}S$. By the argument directly preceding this proposition, there is a coherent functor as above, $F:{\rm Mod}\mbox{-}S\rightarrow \textbf{Ab}$ such that for all $L\in{\rm Mod}\mbox{-}S$, $FL\neq 0$ implies $M|L$. Since $\langle I\mathcal{D}\rangle ={\rm Mod}\mbox{-}S$, there is $N\in\mathcal{D}$ such that $M|IN$, so $F(IN) \neq 0$.  By our assumption on ${\cal D}$ and the fact that $FI:\mathcal{D}\rightarrow\textbf{Ab}$ is a coherent functor, there is a finite-dimensional $R$-module $N' \in {\cal D}$ such that $FI(N') \neq 0$. So $M|IN'$, as required.
 \end{proof}

We now show that, as $M$ varies over indecomposables of bounded dimension, the ``size" of $\phi$, chosen to be part of an isolating pp-pair for $M$ as above, also is bounded above.  Here, by the size of $\phi$ we could mean just the number of symbols in $\phi$ or, more meaningfully, the number of variables and equations appearing in $\phi$.  Indeed, these bounds arise from a minimal projective presentation of $M$, with the number of variables needed being determined by the number of generators needed for $M$ and the number of equations being determined by the number of relations on those generators needed to present $M$.

\begin{definition}
If $\phi$ is a pp-$n$-formula over $R$ then $c(\phi)$ is the number of existentially quantified variables in $\phi$ and $d(\phi)$ is the number of equations in $\phi$ i.e $\phi$ is of the form
\[\exists y_1,\dots ,y_{c(\phi)} (\overline{x} \, \overline{y}) A=0\] where $A$ is an $(n+c(\phi))\times d(\phi)$ matrix with entries from $R$.
\end{definition}

\begin{lemma}\label{isolatingbound}
Suppose that $M\in{\rm ind}_d\mbox{-}R$. Then there is a pp-$1$-pair $\phi/\psi$ such that for all $L\in{\rm Mod}\mbox{-}R$, $\phi(L)/\psi(L)\neq 0$ if and only if $M|L$ and such that $\phi$ has the form
\[\exists y_1,\dots ,y_d (x,\overline{y})A=0\] where $A$ is a $(1+d)\times(1+e)$ matrix with $e\leq \dim M \cdot \dim R$. That is $c(\phi)=d$ and $d(\phi)\leq d\cdot \dim R +1$.
\end{lemma}

\begin{proof}
Suppose that $M\in{\rm ind}_d\mbox{-}R$.  We saw in the comments preceding \ref{fdinimfd} that $\phi$ may be taken to be a pp formula generating the pp-type of any chosen non-zero element $a\in M$.

Let $b_1,\dots ,b_d$ generate $M$ as a $k$-vector space. Let $H$ be a matrix with entries from $R$ such that if $N$ is a module with generators $e_1,\dots ,e_d$ and $\overline{e}H=0$ then there is a surjective map $t:M\rightarrow N$ sending $b_i$ to $e_i$ for $1\leq i\leq d$ (i.e.~the columns of $H$ generate the $R$-linear relations on the $b_i$). Note that $H$ is a $d\times e$ matrix where $e\leq \dim M \times \dim R$.

Given $a\in M$ we have $a=\overline{b}G$ for some $d\times 1$ matrix $G$. The pp-type of $a$ in $M$ is generated by the pp formula
\[\exists y_1,\dots ,y_d (x,\overline{y})
\left(
                       \begin{array}{cc}
                         1 & 0 \\
                         -G & H \\
                       \end{array}
                     \right)=0.\]
See \cite{PreNBK} p.~21 for details.
\end{proof}

\begin{theorem}\label{functorbd}\marginpar{functorbd}
Let $R,S$ be finite-dimensional $k$-algebras. Suppose $I:{\rm Mod}\mbox{-}R\rightarrow{\rm Mod}\mbox{-}S$ is an interpretation functor such that $\langle I{\rm Mod}\mbox{-}R\rangle ={\rm Mod}\mbox{-}S$. For every $d\in {\mathbb N}$ there is $n_d\in \mathbb{N}$ such that, if $N\in{\rm ind}_d\mbox{-}S$, then there is a pp-$m$-pair $\sigma/\tau$ with $c(\sigma)\leq n_d$ such that for all $M\in{\rm Mod}\mbox{-}R$, $\sigma(M)/\tau(M) \neq 0$ if and only if $N|IM$.
\end{theorem}

\begin{definition}
We say that a pp-pair $\phi/\psi$ is {\bf open} on a module $M$ or that $M$ {\bf opens} $\phi/\psi$ is there is an element in $\phi(M)\backslash\psi(M)$. We say that a pp-pair $\phi/\psi$ is {\bf closed} on a module $M$ if $\phi(M)=\psi(M)$.
\end{definition}

In the proof of \ref{functorbd} we take a pp-pair $\gamma/\delta$ over $S$ and pull it back using the interpretation functor $I:{\rm Mod}\mbox{-}R\rightarrow{\rm Mod}\mbox{-}S$ to a pp-pair $\sigma/\tau$ over $R$ so that if $M\in {\rm Mod}\mbox{-}R$ opens $\sigma/\tau$ then $IM$ opens $\gamma/\delta$. We do this in such a way that $c(\sigma)$ is only dependent on $c(\gamma)$ and $d(\gamma)$.  Each variable will be pulled back to an $m$-tuple of variables, indicated in the proof by overlines.
\begin{proof}[of \ref{functorbd}]
Let $s_1=1,\dots ,s_p$ be a $k$-basis for $S$. Suppose that the interpretation functor $I:{\rm Mod}\mbox{-}R\rightarrow{\rm Mod}\mbox{-}S$ is given by the data $(\phi/\psi;\rho_{s_1},\dots ,\rho_{s_p})$ where $\phi/\psi$ is a pp-$m$-pair.

Let $N\in{\rm ind}_d\mbox{-}S$. We know from \ref{isolatingbound} that there exists a pp-$1$-pair $\gamma/\delta$ isolating $N$ with $\gamma$ of the form
\[\exists y_1,\dots ,y_d \bigwedge_{i=1}^{e}xb_i+\sum_{j=1}^dy_ja_{ij}=0\] with $e\leq dp+1$.

For any $N'\in{\rm Mod}\mbox{-}S$, $N|N'$ is and only if $\gamma(N')\supsetneq \delta(N')$. For $1\leq i\leq e$ and $1\leq k \leq p$, let $\beta_{i}^k\in k$ be such that $b_i=\sum_{k=1}^ps_k\beta_{i}^k$. For $1\leq i\leq e$, $1\leq j\leq d$ and $1\leq k\leq p$, let $\alpha_{ij}^k\in k$ be such that $a_{ij}=\sum_{k=1}^ps_k\alpha_{ij}^k$.

Let $\sigma(\overline{X})$ be the pp-$m$-formula over $R$ given by
\[\exists_{1\leq j\leq d} \overline{Y}_j\exists_{1\leq k\leq p}\overline{Z}_k\exists_{1\leq i\leq e}\overline{W}_i\exists_{\substack{1\leq j\leq d \\ 1\leq k\leq p} }\overline{U}_{jk} \ \phi(\overline{X})\bigwedge_{k=1}^p\phi(\overline{Z}_k)\bigwedge_{1\leq j\leq d}\phi(\overline{Y}_{j})\bigwedge_{\substack{1\leq j\leq d \\ 1\leq k\leq p}}\phi(\overline{U}_{jk})\]
\[\bigwedge_{k=1}^p\rho_{s_k}(\overline{X},\overline{Z}_k)\bigwedge_{\substack{1\leq j\leq d \\ 1\leq k\leq p}}\rho_{s_k}(\overline{Y}_{j},\overline{U}_{jk})\bigwedge_{i=1}^e\sum_{k=1}^p\overline{Z}_k\beta_i^k+\sum_{j=1}^d\sum_{k=1}^p\overline{U}_{jk}\alpha_{ij}^k=\overline{W}_i\bigwedge_{i=1}^e\psi(\overline{W}_i)\]

Now for all $M\in{\rm Mod}\mbox{-}R$, an $m$-tuple $\overline{a}$ of elements from $M$ is in $\sigma(M)$ if and only if $\overline{a}+\psi(M) \in \gamma(IM)$.

In exactly the same way, we may find a pp-$m$-formula $\tau$ over $R$ such that $\tau\rightarrow\sigma$ and for all $M\in{\rm Mod}\mbox{-}R$, an $m$-tuple $\overline{a}$ of elements from $M$ is in $\tau(M)$ if and only if $\overline{a}+\psi(M) \in \delta(IM)$.

Thus $IM$ opens $\gamma/\delta$ if and only if $M$ opens $\sigma/\tau$.

Each of the tuples $\overline{Y}_j,\overline{Z}_k,\overline{W}_i$ and $\overline{U}_{jk}$ is of length $m$. Thus the pp-$m$-formula $\sigma(\overline{X})$ has
\[m(d+p+e+dp)+(1+p+d+dp)c(\phi)+\sum_{k=1}^pc(\rho_{s_k})+\sum_{\substack{1\leq j\leq d \\ 1\leq k\leq p}}c(\rho_{s_k})+ec(\psi)\] bound variables. Set
\[n_d:=m(d+p+(dp+1)+dp)+(1+p+d+dp)c(\phi)+(d+1)\sum_{k=1}^pc(\rho_{s_k})+(dp+1)c(\psi).\] Now since $e\leq dp+1$ we have
\[n_d\geq m(d+p+e+dp)+(1+p+d+dp)c(\phi)+\sum_{k=1}^pc(\rho_{s_k})+\sum_{\substack{1\leq j\leq d \\ 1\leq k\leq p}}c(\rho_{s_k})+ec(\psi)\] and $n_d$ is only dependent on $d$.

\end{proof}

Suppose $\theta(\overline{x},\overline{y})$ is a quantifier-free pp formula over $R$ (i.e. a conjunction of $R$-linear equations), that $\overline{x}$ is an $m$-tuple of variables and $\overline{y}$ is an $n$-tuple of variables. Let $\sigma(\overline{x}):=\exists \overline{y} \theta(\overline{x},\overline{y})$. If $\sigma/\tau$ is a pp-pair and $M$ opens $\sigma/\tau$, say $\overline{a}\in \sigma(M) \setminus \tau(M)$ then choose a tuple $\overline{b}$, of elements from $M$, such that $\theta(\overline{a},\overline{b})$ holds in $M$ and let $M'$ be the $R$-submodule of $M$ generated by the entries of $\overline{a}$ and $\overline{b}$. Then clearly $M'$ opens $\sigma/\tau$ and $\dim M'\leq (m+n)\dim R$.

\begin{cor}\label{boundingdimneededtohitfds}
Suppose $I:{\rm Mod}\mbox{-}R\rightarrow{\rm Mod}\mbox{-}S$ is an interpretation functor such that $\langle I{\rm Mod}\mbox{-}R\rangle ={\rm Mod}\mbox{-}S$. For each $d\in\mathbb{N}$ there exists $b_d\in\mathbb{N}$ such that for every $M\in{\rm ind}_d\mbox{-}S$ with $\dim M\leq d$ there is $N\in{\rm mod}\mbox{-}R$ with $\dim N\leq b_d$ such that $M$ is a direct summand of $IN$.
\end{cor}
\begin{proof}
Suppose $I:{\rm Mod}\mbox{-}R\rightarrow{\rm Mod}\mbox{-}S$ is an interpretation functor given by $(\phi/\psi;\rho_{s_1},\dots ,\rho_{s_p})$ where $\phi/\psi$ is a pp-$m$-pair. Let $b_d=(n_d+m)\dim R$ where $n_d$ is chosen as in \ref{functorbd}. Let $M\in{\rm ind}_d\mbox{-}S$. By that result there is a pp-$m$-pair $\sigma/\tau$ with $c(\sigma)\leq n_d$ such that for all $L\in {\rm Mod}\mbox{-}R$, $L$ opens $\sigma/\tau$ if and only if $M$ is a direct summand of $IL$.

Suppose $L\in {\rm Mod}\mbox{-}R$ opens $\sigma/\tau$. Since $\sigma$ is a pp-$m$-formula with $c(\sigma)\leq n_d$ there is, by the argument directly preceding this corollary, a submodule $N$ of $L$ opening $\sigma/\tau$ of dimension less that or equal to $b_d$. Thus there is an $R$-module $N$ with $\dim N\leq b_d$ such that $M$ is a direct summand of $IN$.
\end{proof}

In \cite{KraSpec}, a noetherian algebra is said to be {\bf endofinitely tame} if for each $n\in\mathbb{N}$ the Ziegler closure of the set of indecomposable finitely presented modules of endolength $\leq n$ contains only finitely many non-finitely presented points. A finite-dimensional algebra over an algebraically closed field is tame if and only if it is endofinitely tame \cite[Theorem 10.13]{KraSpec}.

\begin{theorem}\label{krausetransendtame}\cite[Corollary 10.8]{KraSpec}
A noetherian algebra $S$ is endofinitely tame if and only if
for every $n \in \mathbb{N}$ there is an endofinitely tame noetherian algebra $R$, some $m\in \mathbb{N}$,
and an interpretation functor $F: {\rm Mod}\mbox{-}R\rightarrow {\rm Mod}\mbox{-}S$ such that all but finitely many modules
in ${\rm ind}_n\mbox{-}S$ occur as a direct summand of some finitely presented module $FM$ with
$M \in \bigcup_{i=1}^m {\rm ind}_i\mbox{-}R$.
\end{theorem}

\begin{theorem}\label{onlywildintwild}
Let $R,S$ be finite-dimensional algebras over an algebraically closed field. Suppose that $I:{\rm Mod}\mbox{-}R\rightarrow {\rm Mod}\mbox{-}S$ is an interpretation functor such that $\langle {\rm Mod}\mbox{-}R\rangle={\rm Mod}\mbox{-}S$. If $S$ is wild then $R$ is wild.
\end{theorem}
\begin{proof}
Suppose $R$ is tame and $I:{\rm Mod}\mbox{-}R\rightarrow {\rm Mod}\mbox{-}S$ is an interpretation functor such that $\langle {\rm Mod}\mbox{-}R\rangle={\rm Mod}\mbox{-}S$. Since $R$ is tame, $R$ is endofinitely tame. By \ref{boundingdimneededtohitfds}, for each $n\in\mathbb{N}$ there is a $b_n$ such that if $L\in{\rm ind}_n\mbox{-}S$ then there is an $M\in \bigcup_{i=1}^{b_n} {\rm ind}_i\mbox{-}R$ such that $L$ is a direct summand of $FM$. Thus by \ref{krausetransendtame}, $S$ is endofinitely tame. Hence $S$ is tame.
\end{proof}

\section{Domestic does not interpret non-domestic}

In this section we show that if $R,S$ are finite-dimensional $k$-algebras with $k$ an algebraically closed field and $I:{\rm Mod}\mbox{-}R\rightarrow{\rm Mod}\mbox{-}S$ is an interpretation functor such that $\langle I{\rm Mod}\mbox{-}R\rangle={\rm Mod}\mbox{-}S$ then, if $R$ is tame domestic, so also is $S$.  We allow ``finite type" to be included in ``tame" since, if $R$ is of finite representation type, so is $S$, for example by the following general result.

\begin{proposition} (\cite[15.1]{PreMAMS})  If $I:{\cal C} \rightarrow {\cal D}$ is an interpretation functor then for every indecomposable pure-injective $M$ in $\langle I{\cal C}\rangle$ there is an indecomposable pure-injective $N\in {\cal D}$ such that $M$ is a direct summand of $IN$.
\end{proposition}

\begin{definition}
The {\bf endolength} $\text{el(M)}$ of a module $M$ is the length of $M$ as a module over $\text{End}(M)$. We say that an indecomposable module is {\bf generic} if it is of finite endolength but not finitely presented.
\end{definition}

\begin{lemma}
Let $R,S$ be finite-dimensional $k$-algebras with $k$ algebraically closed. Let $I:{\rm Mod}\mbox{-}R\rightarrow {\rm Mod}\mbox{-}S$ be an interpretation functor. If $G\in{\rm Mod}\mbox{-}R$ is generic then $IG$ is either zero or finite endolength. If $R$ is tame then $IG$ is not finite-dimensional.
\end{lemma}
\begin{proof}
We know that if $M$ is finite endolength then $IM$ is also finite endolength (see for instance \cite[Cor.~9.7]{KraSpec}).  If $G$ is generic and $R$ is tame then $\text{End}_R G/\text{rad} \text{End}_RG\cong k(x)$ \cite[4.4]{CBtamegen} and the embedding of $\text{rad}\text{End}_R G$ into $\text{End}_R G$ is split. Thus $IG$ is a quotient of two $k(x)$-modules, so is either zero or infinite-dimensional.
\end{proof}

\begin{lemma} (cf. \cite[9.6]{HerzogloccohZg}) Let $R$ be a finite-dimensional $k$-algebra.
Any compact subset of $\text{Zg}_R$ not containing any generic module contains only finitely many modules of each finite dimension over $k$.
\end{lemma}
\begin{proof}
For any topological space, the intersection of a compact subset with a closed subset is compact. Thus if $Y$ is a compact subset of $\text{Zg}_R$ then for any $n\in\mathbb{N}$, the intersection of $Y$ with the closed subset of points of finite endolength $\leq n$ is compact. If it were infinite then it would contain a non-isolated point but any non-isolated point of finite endolength is generic. Thus $Y$ contains only finitely many points of each finite dimension over $k$.
\end{proof}

We now recall some basic information about Ziegler spectra of Dedekind domains all of which is contained in \cite[Section 5.2.1]{PreNBK}.

Let $R$ be a Dedekind domain. The points of $Zg_R$ are as follows:
the field of fractions, $Q$, of $R$; for each prime $\mathfrak{p}$, a Pr\"ufer point $R_{\mathfrak{p}^\infty}$ and an adic point $\overline{R_\mathfrak{p}}$; for each prime $\mathfrak{p}$ and for each positive integer $n$, the finite-length module $R/\mathfrak{p}^n$.  Each point $R/\mathfrak{p}^n$ is clopen and $Q$ belongs to the closure of each infinite-dimensional point.

\begin{remark}\label{dedekindgenericinopen}
Suppose that $R$ is a Dedekind domain and $U$ is an open subset of $\text{Zg}_R$ containing the field of fractions of $R$. Then the complement of $U$ is finite.

This follows since the complement of $U$, being closed, can contain only finite-dimensional, hence isolated, points (for the generic is in the closure of each adic and Pr\"{u}fer point), each adic and Pr\"ufer point is in $U$. But this complement also is compact, so must be finite.
\end{remark}

We will heavily use the following result of Crawley-Boevey.

\begin{theorem}\label{CBgenc}\marginpar{CBgenc}  \cite[p.~2, 5.2, 5.4]{CBtamegen}
If $R$ is a finite-dimensional algebra over an algebraically closed field $k$ and $R$ has tame representation type, then for each generic module $G$ there is $f\in k[x]$ and a $(k[x,f^{-1}],R)$-bimodule $B_G$ such that the following hold.
\begin{enumerate}
\item  As a left $k[x,f^{-1}]$ module, $B_G$ is free of rank equal to the endolength of $G$ and $k(x) \otimes B_G\simeq G$.
\item  The functor $-\otimes B_G$ from $k[x,f^{-1}]$-modules to $R$-modules reflects isomorphism and preserves indecomposability and Auslander-Reiten sequences.
\item For each $d\in\mathbb{N}$,  all but finitely many indecomposable $R$-modules of dimension $d$ have the form $k[x,f^{-1}]/\langle g\rangle\otimes B_G$ for some generic $G$ and some $g \in k[x,f^{-1}]$.
\end{enumerate}
\end{theorem}

\begin{theorem}\label{nondom1}\marginpar{nondom1}
Suppose that $k$ is an algebraically closed field and let $R, S$ be finite-dimensional $k$-algebras with $R$ domestic.  If there is an interpretation functor $I:{\rm Mod}\mbox{-}R\rightarrow {\rm Mod}\mbox{-}S$ such that $\langle I{\rm Mod}\mbox{-}R\rangle={\rm Mod}\mbox{-}S$ then $S$ must be domestic.
\end{theorem}
\begin{proof}
Since $R$ is domestic, by \ref{onlywildintwild}, $S$ must be tame.  Let $G_1,\dots ,G_n$ be the generic $R$-modules.  As noted already, each $IG_i$ is of finite endolength, hence the closed subset
\[X:=\langle IG_1,\dots ,IG_n\rangle\cap\text{Zg}_S\]
is finite. Thus, since $S$ is non-domestic, there is a generic $S$-module $H$ of endolength strictly greater than the endolength of any module in $X$. Since $X$ is closed and $H\notin X$, there is a basic Ziegler-open set containing $H$ and not intersecting $X$, so there is a pp-defined = $\varinjlim$-commuting extension, $F$, of a finitely presented functor from ${\rm Mod}\mbox{-}S$ to $ {\bf Ab}$ such that $FH\neq 0$ and $FM=0$ for all $M\in X$.

\vspace{4pt}

\noindent\textbf{Claim 1}: For each $m\in\mathbb{N}$, the set $\{N\in\text{Zg}_R \ \vert \  FIN\neq 0  \text{ and } \text{el} N\leq m\}$ is finite.

This set is compact since the condition $FI(-) \neq 0$ defines a compact open subset of ${\rm Zg}_R$ and the set of points of endolength $\leq m$ is closed.  But this set contains no generic point, since $FIG_i = 0$ for each $i$, hence it is finite.

\vspace{4pt}

\noindent\textbf{Claim 2}: There is $d\in \mathbb{N}$ such that the set $\{C\in {\rm ind}\mbox{-}S \ \vert \  FC\neq 0\}$ contains infinitely many modules of dimension $d$.

Let $A_H:=k[x,f^{-1}]$ and let $B_H$ be as in \ref{CBgenc}. Since $-\otimes B_H:{\rm Mod}\mbox{-}A_H\rightarrow {\rm Mod}\mbox{-}S$ is an interpretation functor (\ref{repembinterp}), the composition, $F'$, of this with $F$ is a finitely presented functor i.e. is given by a pp-pair. Since $F(H)\neq 0$, we have $F'(k(x))\neq 0$. Thus, by Remark \ref{dedekindgenericinopen}, $F'((A_H/\mathfrak{p})\otimes B_H)\neq 0$ for almost all non-zero primes $\mathfrak{p}$ of $ A_H$.  Since each factor $A_H/\mathfrak{p}$ is 1-dimensional, $\dim_k\big((A_H/\mathfrak{p})\otimes B_H\big)=\text{el} H$. Since $-\otimes B_H$ preserves isomorphism classes, there are, therefore, infinitely many pairwise non-isomorphic indecomposable $S$-modules $C$ of dimension $\text{el} H$ with $FC\neq 0$.

To finish the proof of the theorem, let $d\in \mathbb{N}$ be such that the set $\{C\in {\rm ind}\mbox{-}S \ \vert \  FC\neq 0\}$ contains infinitely many modules of dimension $d$. Let $b_d\in\mathbb{N}$ be such that for all $C\in{\rm ind}\mbox{-}S$ with $\dim C\leq d$ there exists $N\in{\rm ind}\mbox{-}S$ with $\dim N\leq b_d$ such that $C|IN$ (\ref{boundingdimneededtohitfds}).  By Claim 1, there are only  finitely many $N\in{\rm ind}\mbox{-}S$ with $\dim N\leq b_d$ and $FIN\neq 0$, hence only finitely many direct summands of these, contradiction.

\end{proof}





\begin{thebibliography}{99}
\bibitem{PreBur} K. Burke and M. Prest, The {Z}iegler and {Z}ariski spectra of some domestic string algebras, Algebr. Represent. Theory, 5(3) (2002), 211-234.

\bibitem{CohNew} P. M. Cohn, Free Ideal Rings and Localization in General Rings, New Mathematical Monographs, Vol. 3, Cambridge University Press, 2006.

\bibitem{CBtamegen} W. Crawley-Boevey, Tame algebras and generic modules, Proc. London Math. Soc. (3), 63(2), 241-265.

\bibitem{Eilen}  S. Eilenberg, Abstract description of some basic functors, J. Indian Math. Soc., 24 (1960), 231-234.

\bibitem{Han} Y. Han, Controlled wild algebras, Proc. London Math. Soc. (3), 83(2) (2001), 279-298.

\bibitem{Har} R. Harland, Pure-injective Modules over Tubular Algebras and String Algebras, Doctoral Thesis, University of Manchester, 2011, {\it available at} www.maths.manchester.ac.uk/$\sim$mprest/publications.html.

\bibitem{HerzogloccohZg} I. Herzog, The {Z}iegler spectrum of a locally coherent {G}rothendieck
              category, Proc. London Math. Soc. (3), 74(3) (1997), 503-558.

\bibitem{KrExactDef}  H. Krause, Exactly definable categories, J. Algebra, 201(2) (1998), 456-492.

\bibitem{KraSpec} H. Krause, The spectrum of a module category, Mem. Amer. Math. Vol., 149, 2001.

\bibitem{KraGen} H. Krause, Generic modules over {A}rtin algebras, Proc. London Math. Soc. (3), 76(2) (1998), 276-306.

\bibitem{KP1}  T. G. Kucera and M. Prest, Imaginary modules, J. Symbolic Logic, 57(2) (1992), 698-723.

\bibitem {Len}  H. Lenzing,  Endlich pr\"{a}sentierbare Moduln,  Arch. Math., 20(3) (1969), 262-266.

\bibitem{Nag}  H. Nagase, Non-strictly wild algebras, J. London Math. Soc. (2), 67(1) (2003), 57-72.

\bibitem{PoPr} F. Point and M. Prest,  Decidability for theories of modules,
J. London Math. Soc. (2), 38(2) (1988), 193-206.

\bibitem{PreBk}  M. Prest,  Model Theory and Modules,  London Math. Soc.
Lect. Note Ser., Vol. 130, Cambridge University Press, 1988.

\bibitem{PreRepn}  M. Prest,  Representation embeddings and the Ziegler spectrum,  J. Pure Appl. Algebra, 113(3) (1996), 315-323.

\bibitem{PreEpi}  M. Prest,  Epimorphisms of rings, interpretations of modules and strictly wild algebras,  Comm. Algebra, 24(2) (1996), 517-531.

\bibitem{PreInterp}  M. Prest, Interpreting modules in modules, Ann. Pure Applied Logic, 88(2-3) (1997), 193-215.

\bibitem{PreNBK} M. Prest, Purity, Spectra and Localisation, Encyclopedia of Mathematics and its Applications, Vol. 121, Cambridge University Press, 2009.

\bibitem{PreMAMS} M. Prest, Definable additive categories: purity and model theory,  Mem. Amer. Math. Soc., Vol. 210, 2011.

\bibitem{PreAbex} M. Prest, Abelian categories and definable additive categories, 2012, arXiv:1202.0246.

\bibitem{PreRajStrShv}  M. Prest and R. Rajani, Structure sheaves of definable additive categories, J. Pure Applied Algebra, 214 (2010), 1370-1383.

\bibitem{SchrKGdim} J. Schr\"{o}er, On the Krull-Gabriel dimension of an algebra, Math. Z., 233(2) (2000), 287-303.

\bibitem{SiSk3}  D. Simson and A. Skowro\'{n}ski, Elements of the Representation Theory of Associative Algebras.  3:
Representation-Infinite Tilted Algebras, London Math. Soc. Student Texts,  Vol. 65, Cambridge University Press, 2006.

\bibitem{Watt}  C. E. Watts,  Intrinsic characterisation of some additive functors,  Proc. Amer. Math. Soc., 11(1) (1960), 5-8.

\bibitem{Zie}  M. Ziegler,  Model theory of modules,  Ann. Pure Appl.
Logic, 26(2) (1984), 149-213.

\end{thebibliography}
\end{document}